\documentclass[11pt,reqno]{amsart}
\usepackage[top=3cm, left=3cm, a4paper]{geometry}
\usepackage{amsfonts,amscd,latexsym,amsmath,amssymb,enumerate,verbatim,amsthm,epsfig,alltt,cancel}
\theoremstyle{plain}
\newtheorem{master}{Master}[section]

\newtheorem{thm}[master]{Theorem}

\newtheorem{cor}[master]{Corollary}

\newcommand{\Nat}{\mathbb{N}}

\newcommand{\Int}{\mathbb{Z}}

\theoremstyle{definition}

\theoremstyle{remark}

\numberwithin{equation}{section}
\begin{document}
\title{Embeddings into monothetic groups}
\author{Michal Doucha}
\address{Institute of Mathematics, Academy of Sciences, Prague, Czech republic}
\email{doucha@math.cas.cz}
\begin{abstract}
We provide a very short elementary proof that every bounded separable metric group embeds into a monothetic bounded metric group, in such a way that the result of Morris and Pestov that every separable abelian topological group embeds into a monothetic group is an immediate corollary. We show that the boundedness assumption is essential.
\end{abstract}
\maketitle
In \cite{MP2} Morris and Pestov prove that every separable abelian topological group embeds into a monothetic group (following their previous generalization of the Higman-Neumann-Neumann theorem for topological groups from \cite{MP1}). Since their proof is rather long and uses several non-elementary results we provide here a short elementary proof of their result which is actually a generalization: it is in the category of metric groups.
\begin{thm}
Let $G$ be a separable abelian group with a bounded invariant metric $d$. Then $d$ extends to a (bounded by the same constant) metric $D$ on $G\oplus C$, where $C$ is a cyclic group which is dense in $G\oplus C$. In particular, $G$ embeds into a monothetic metric group.
\end{thm}
\begin{proof}
Instead of metric, we shall work with the corresponding norm, i.e. the distance of an element from the group zero, which we shall still denote $d$, or $D$. That is, $d(g)$ denotes $d(g,0)$. Without loss of generality, suppose that $d$ is bounded by $1$. Let $H=\{h_n:n\in\Nat\}$ be a countable dense subgroup of $G$ and let $\pi:\Nat\rightarrow \Nat^2$ be some bijection; we denote by $\pi(n)(1)$, resp. $\pi(n)(2)$, the respective coordinates of $\pi(n)$. Suppose the cyclic group $C$ is generated by some $c$. By induction, we shall construct a partial norm $D'$ on $H\oplus C$ which extends $d$, i.e. a partial function satisfying $D'(g)=0$ iff $g=0$, $D'(g)=D'(-g)$ and $D'(g_1+\ldots+g_n)\leq D'(g_1)+\ldots+D'(g_n)$, whenever the corresponding elements are in the domain of $D'$. Moreover, we shall produce an increasing sequence $(k_n)_n$ such that $D'(c^{k_n}-h_{\pi(n)(1)})\leq 1/\pi(n)(2)$. At the end, we may define $D$ by $D(x)=\min\{1,\inf\{\sum_{i=1}^m D'(x_i):x=\sum_{i=1}^m x_i,(x_i)_i\subseteq \mathrm{dom}(D')\}\}$, for $x\in H\oplus C$. Since $D'$ was a partial norm, $D$ extends $D'$, thus it extends $d$, and by the induction we will have guaranteed that $C$ is dense in $H\oplus C$.

At the first step of the induction, we set $D$ to be equal to $d$ on $H$, set $k_1=1$ and we define $D'(c-h_{\pi(1)(1)})=D'(h_{\pi(1)(1)}-c)=1/\pi(1)(2)$. $D'$ is clearly a partial norm. Suppose we have done the first $n-1$ steps, found $k_1,\ldots,k_{n-1}$ and defined $D'$ appropriately so that it is a partial norm. Let $\delta=\min\{1/\pi(i)(2):i<n\}$ and let $k_n>k_{n-1}$ be arbitrary satisfying $k_n>k_{n-1}/\delta$. Then we set  $D'(c^{k_n}-h_{\pi(n)(1)})=D'(h_{\pi(n)(1)}-c^{k_n})=1/\pi(n)(2)$. We claim that $D'$ is still a partial norm. Suppose on the contrary that the triangle inequality is broken, i.e. there are $x,x_1,\ldots,x_n\in\mathrm{dom}(D')$ such that $x=x_1+\ldots+x_n$ and $D'(x)>D'(x_1)+\ldots+D'(x_n)$. We shall suppose that $x\in H$, that is the most important case and the other case is treated analogously. For any $z\in H\oplus C$, denote by $k(z)$ the unique integer such that $z$ can be written as $h\oplus c^{k(z)}$. Since $k(x)=0$, we must have $\sum_{i=1}^n k(x_i)=0$. For at least one $i\leq n$ we must have that $k(x_i)$ is $k_n$ or $-k_n$, since before the extension at the $n$-th step, $D'$ was a partial norm. Also, since $H$ is abelian, we may suppose that for no $i,j\leq n$ we have $k(x_i)=-k(x_j)\neq 0$, since in that case $x_i+x_j=0$. Let $I\subseteq \{i\leq n:0\neq k(x_i)\neq k_n\}$. It follows that $|\sum_{i\in I} k(x_i)|\geq k_n$, so by definition of $k_n$ we have $|I|>1/\delta$, so $\sum_{i\in I} D'(x_i)>1$, a contradiction.
\end{proof}
Notice that the numbers $(k_n)_n$ were chosen completely independently of the metric/norm, and the same sequence may be used for any metric/norm bounded by $1$. Secondly, notice that there is no change in the proof if we replace metric, resp. norm by pseudometric, resp. pseudonorm.
\begin{cor}[Morris, Pestov]
Every separable abelian topological group $G$ embeds into a monothetic group.
\end{cor}
\begin{proof}
Let $H$ be a countable dense subgroup of $G$ and let $(\rho_\alpha)_{\alpha<\kappa}$ be a collection of pseudonorms bounded by $1$ which give the topology of $G$, and $H$. By the previous proof we may extend each $\rho_\alpha$ on $H\oplus C$, where $C$ is cyclic, using the same numbers $(k_n)_n$ and $\pi:\Nat\rightarrow\Nat^2$. By either completing $H\oplus C$ with respect to these pseudonorms or extending the pseudonorms on $G\oplus C$ by amalgamation, we obtain a monothetic group to which $G$ embeds.
\end{proof}
One may wonder why we assume boundedness of the metric/norm. It turns out it is essential. Consider $\ell^2_1$, the two-dimensional Banach space with $\ell_1$ norm, and let $H$ be its metric subgroup generated by the two basis elements $e_1,e_2$. Suppose there is a norm $D$ on $H\oplus C$, with $C$ cyclic, which extends the norm on $H$, and $C$ is dense in $H\oplus C$. Then there exist $n,m\in\Int$ such that $D(c^n-e_1)<1/2$ and $D(e_2-c^m)<1/2$. Then however $|m|+|n|=D(-m\cdot e_1+n\cdot e_2)=D(c^{mn}-m\cdot e_1+n\cdot e_2-c^{nm})\leq mD(c^n-e_1)+nD(e_2-c^m)<1/2(|m|+|n|)$, a contradiction.\\

\noindent {\bf Problem.} Prove a metric version of the Higman-Neumann-Neumann theorem. Either for separable groups with general left-invariant metrics, or for groups with bi-invariant metric.

\end{document}